\documentclass{amsart}

\usepackage[T1]{fontenc}
\usepackage[utf8]{inputenc} 
\usepackage{lmodern}

\usepackage{amsmath,amssymb,amsfonts,mathtools}

\usepackage{tikz}

\usepackage{hyperref}

\newcommand{\R}{\mathbb{R}}

\newcommand{\Sph}{\mathbb{S}}
\newcommand{\B}{\mathbb{B}}

\newcommand{\width}{\operatorname{width}}
\newcommand{\rw}{\operatorname{rw}}

\theoremstyle{plain}
\newtheorem{theorem}{Theorem}
\newtheorem{lemma}[theorem]{Lemma}
\newtheorem{claim}[theorem]{Claim}
\newtheorem{conjecture}[theorem]{Conjecture}

\theoremstyle{definition}

\theoremstyle{remark}
\newtheorem*{remark}{Remark}

\title{A note on the~affine~plank~conjecture}

\date{}

\author
{Egor Bakaev}
\address{Department of Computer Science, University of Copenhagen}

\author
{Amir Yehudayoff}
\address{Department of Computer Science, University of Copenhagen,
	and Department of Mathematics, Technion-IIT}

\subjclass[2020]{Primary 52A40; Secondary 52C17}

\begin{document}
	
\begin{abstract}
	In 1951, Bang posed the \emph{affine plank conjecture}, which remains open: If a convex body in $\mathbb{R}^d$ is covered by planks, then the total relative width of the planks is at least one.
	We prove a lower bound of $2/(1+\sqrt{d})$ for this total relative width.
	The best previously known lower bound was $2/(1+d)$.
\end{abstract}

		\maketitle
		
\section*{Introduction}

In 1932, Tarski posed the following natural problem \cite{tarski1932uwagi}, which Bang solved in 1950 \cite{bang1950covering, bang1951solution}.
{\em Is it true that the total width of planks covering a convex body
is at least the width of the body?}

	A plank $P$ is a region between two parallel hyperplanes, i.e., a set of the form
	$
	P=\{x\in\mathbb{R}^d: |\langle x,u\rangle-t|\le \frac{w}{2}\}
	$ where $u\in \Sph^{d-1}$ is the normal, $t\in\mathbb{R}$ is the translation, and $\width (P) := w$.
Let $K \subset \R^d$ be a convex body (that is, a compact convex set with non-empty interior). 
The support function of $K$ is	
\[
	\R^d \ni x \mapsto h_K(x)=\max\{\langle x,y\rangle: y\in K\}.
	\]	
	The width of $K$ in direction $u \in \Sph^{d-1}$ is
	\[
	w_K(u)=h_K(u)+h_K(-u).
	\]
The width of $K$ is $\width(K) = \min_{u \in \Sph^{d-1}} w_K(u)$.

\begin{theorem}[Tarski's plank problem / Bang's theorem]
	\label{th:bang}
	If the planks $P_1, \dots, P_m \subset \R^d$ cover a convex body $K\subset\mathbb{R}^d$ then
	\[
	\sum_{i=1}^m \width (P_i) \ge \width(K).
	\]
\end{theorem}

In his 1951 paper, Bang also posed the following natural extension.
The relative width of the plank $P$ with respect to $K$ is
\[
	\rw_K(P) := \frac{\width(P)}{w_K(u)},
\]
where $u$ is the normal of $P$.
Unlike the width, the relative width is invariant under invertible affine maps.

	\begin{conjecture}[Bang's affine plank conjecture]
		\label{conj}
		If the planks $P_1,\dots,P_m \subset \R^d$ cover a convex body $K\subset\mathbb{R}^d$ then
		\[
		\sum_{i=1}^m \rw_K(P_i)
		 \ge 1.
		\]
	\end{conjecture}
	
	This conjecture remains open. In fact, it is open even in the plane (some sources incorrectly claim that the conjecture has been resolved in the plane, citing \cite{bang1953some}; however, that paper proves the conjecture only for the special case of two planks).
	In 1991, Ball proved the conjecture for symmetric convex bodies \cite{ball1991plank}. Ambrus showed \cite{ambrus2010appendix} that proving the conjecture for simplices implies it in full generality. For overviews of the related results, applications and problems, see \cite{bezdek2013tarski,brass2005research} and more recent  \cite{verreault2022survey,toth2023lagerungen}.

John's theorem and Bang's theorem imply that, in the setting of Conjecture~\ref{conj}, the sum of the relative widths is at least $\frac{1}{d}$ (see~\cite{chambers2016note}
and \cite{bezdek2009covering} for a different proof). 
Ball's solution of the symmetric case and a result of Minkowski and Radon
 (see, e.g. \cite[Corollary 1.4.2]{toth2015measures})
yield the better lower bound $\frac{2}{1+d}$ (see \cite[Theorem 2 and Remark 5]{chambers2016note}). Our result is Theorem \ref{th:two-over-sqrtd}, which gives the lower bound~$2/(1+\sqrt{d})$.

\subsection*{Acknowledgments}
We thank Alexander Polyanskii and Ansgar Freyer for helpful remarks. 
The authors are supported by the DNRF Chair grant 20231101-28697.
The authors are part of BARC, Basic Algorithms Research Copenhagen, supported by the VILLUM Foundation grant 54451.

\section*{The lower bound}

For the rest of this text, fix a convex body $K \subset \R^d$
and consider the centrally symmetric body 
\[
L:=\tfrac12(K-K).
\]
\begin{lemma} 
		For any $u \in \Sph^{d-1}$, we have $w_L(u)=w_K(u).$
\end{lemma}
\begin{proof} 
$	w_L(u)=2h_L(u)=h_{K-K}(u) = h_K(u)+h_K(-u)=w_K(u).$
\end{proof} 

The problem is affine-invariant, so we may assume that $L$ is in John's position (see, for example, \cite[Lecture 3]{ball1997elementary}).
In other words, let $T$ be an invertible linear map so that $TL$ is in John's position,
and replace the body~$K$ and the planks $P_i$ by $TK$ and $TP_i$. 	
Now, for the unit ball $\B$ we have
\[
\B \subseteq L\subseteq \sqrt{d}\B.
\]
This already gives an improvement over the $2/(1+d)$ lower bound.

\begin{claim}
			If the planks $P_1,\dots,P_m \subset \R^d$ cover 
			a convex body $K\subset\mathbb{R}^d$ then
			\[
			\sum_{i=1}^m \rw_K(P_i) \ge \frac1{\sqrt d}.
			\]
\end{claim}

	\begin{proof}
Since $w_K(u) = w_L(u)$
for every $u \in \Sph^{d-1}$, it follows that
$2 \leq w_K(u) \leq 2 \sqrt{d}$.
By Bang's Theorem~\ref{th:bang},
		\[
		\sum_i \rw_K(P_i) 
		\ge \frac{\sum_i \width (P_i)}{2\sqrt d}
		\ge \frac{\width(K)}{2\sqrt d}
		\ge \frac{2}{2\sqrt d}.
		\qedhere
		\]
	\end{proof}

Next, we prove a stronger lower bound. It was observed (see, e.g., \cite{gardner1988relative}) that Bang's proof of Theorem \ref{th:bang} actually yields the stronger statement (Lemma \ref{lemma:chord} below). The proof exactly follows the proof of Lemma 2.3 in \cite{verreault2022survey}.

Denote by $[a,b]$ the line segment with the endpoints $a,b \in \R^d$.
Define the \textit{chord length} of a convex body $M \subset \R^d$ in direction $u \in \Sph^{d-1}$ as
\[
\ell_M(u) := \sup\{t\ge 0: \exists\  x,y \in M: x-y=tu\}.
\]

\begin{remark}
	It is easy to see that $\ell_M(u) \leq w_M(u)$.
	Moreover, a chord realizing $\ell_M(u)$ is an \textit{affine diameter}, that is,
	its endpoints lie in two parallel supporting hyperplanes of~$M$.
	Consequently, 
	$
	\width(M) \leq \ell_M(u).
	$
	For a proof see, e.g., Lemma~2.3 in~\cite{verreault2022survey}.
\end{remark}

 \begin{lemma}
 	\label{lemma:chord}
If the planks $P_1,\dots,P_m$ with unit normals $u_1,\dots, u_m$ cover 
the convex body $K\subset\R^d$, then
\[
\sum_{i=1}^m \frac{\width(P_i)}{\ell_K(u_i)} \ge 1.
\]
\end{lemma}

The {\em radial function} of a centrally symmetric convex body $M\subset\R^d$ is
\[
\Sph^{d-1} \ni u \mapsto \rho_M(u) := \sup\{t\ge 0: tu\in M\}.
\]

\begin{lemma}
	\label{lem:chord-L}
	For every $u\in\Sph^{d-1}$, we have $\ell_K(u)=\ell_L(u)=2\rho_L(u).$
\end{lemma}

\begin{proof}
	Observe that
	\begin{align*}
	tu\in L=\tfrac12(K-K) \quad \Longleftrightarrow \quad
	\exists  \ x,y\in K \ x-y=2tu .
	\end{align*}
By taking the supremum, $\ell_K(u)=2\rho_L(u)$.
	As $L$ is centrally symmetric and convex, $\ell_L(u) = 2\rho_L(u)$.
\end{proof}

\begin{lemma}
	\label{lem:john-radial}
	For every $u\in\Sph^{d-1}$,
	\[
	\frac{\ell_K(u)}{w_K(u)}  =
	\frac{\ell_L(u)}{w_L(u)} \ge \frac{2}{1+\sqrt d}.
	\]
\end{lemma}

Before proving the technical lemma, let us see how it yields the following lower bound.

\begin{theorem}
	\label{th:two-over-sqrtd}
			If the planks $P_1,\dots,P_m \subset \R^d$ cover 
			a convex body $K\subset\mathbb{R}^d$ then
	\[
	\sum_{i=1}^m \rw_K(P_i) \ge \frac{2}{\sqrt d+1}.
	\]
\end{theorem}

\begin{proof}
	 Let $u_1,\dots, u_m$ be the unit normals of the planks. By Lemmas \ref{lem:john-radial} and \ref{lemma:chord}, 
	\[
	\sum_{i=1}^m\rw_K(P_i)
	= \sum_{i=1}^m \frac{\width(P_i)}{w_K(u_i)}
	 \ge \frac{2}{1+\sqrt d}\cdot \sum_{i=1}^m \frac{\width(P_i)}{\ell_K(u_i)}
	 \ge \frac{2}{1+\sqrt d}.
	 \qedhere 
	\]
\end{proof}

\begin{figure}
	\centering
	\includegraphics[]{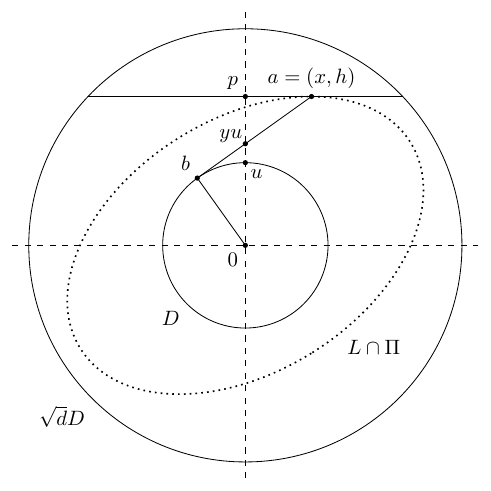}
	\caption{}
	\label{fig:circle}
\end{figure}

\begin{proof}[Proof of Lemma~\ref{lem:john-radial}]
	The equality follows from the facts $\ell_K(u)=\ell_L(u)$ and $w_K(u)=w_L(u)$ that we already established. It remains to prove the inequality.
	Set $h:=h_L(u)$. Since $\B\subseteq L\subseteq \sqrt d\,\B$, we have $1\le h\le \sqrt d$.
	Choose $a\in \partial L$ with $\langle a,u\rangle=h$, and set $\Pi$ to be the vector space spanned by $u,a$.
	If $a = t u$ for some $t \in \R$ then 
	\[\frac{\ell_L(u)}{w_L(u)} =1 \ge \frac{2}{1+\sqrt d}.\]
	So we can assume $\Pi$ is two-dimensional.
	As $\ell_L(u)=2\rho_L(u)$ and $w_L(u)=2h$, 
	it suffices to prove
	\[
	\frac{\rho_{L\cap \Pi }(u)}{h}
	\ge \frac{2}{1+\sqrt d}.
	\]
	Identify $\Pi \simeq\R^2$ and write (without loss of generality)
	\[
u=(0,1),\qquad 	a=(x,h),\qquad x,h > 0,\qquad r:=\|a\|=\sqrt{x^2+h^2}\le \sqrt d,
	\]
	see Figure~\ref{fig:circle}.
It holds that $r>h\geq 1$. 

Let $D$ be the unit disc in $\R^2$.
	Let $b\in\partial D$ be the tangency point of the segment from $a$ to $D$ such that segment $[a,b]$ is ``above $D$''; that is, the point of intersection of $[a,b]$ with the $u$-axis
is $yu$ for $y > 0$ (the fact that $y>0$ relies also on the choice of $a$).
If $y=1$, we may assume $a=u$; this case has already been considered. Hence, we   assume $y>1$.
	The segment $[a,b]$ lies in $L\cap\Pi$, so $\rho_{L\cap\Pi}(u) \geq y$. Therefore it is enough to show that
	\[
	\frac{y}{h} \geq \frac{2}{1+\sqrt d}.
	\]

Take the point $p=(0,h)$, and consider the two similar right triangles given by $0,b,yu$ and $a,p,yu$. 
Because $\|b\|=1$,
\[
x =  \frac{\|p - a\|}{\|b - 0 \|}
= \frac{\|p - yu\|}{\|b - yu\|}
=\frac{h-y}{\sqrt{y^2-1}}.
\]
Therefore,
\begin{align*}
	& x\sqrt{y^2-1}  = h-y\\
	\Longrightarrow\quad  & x^2\bigl(y^2-1\bigr) = (h-y)^2 \\
	\Longrightarrow\quad  & (x^2-1)y^2+2hy-r^2   = 0.
\end{align*} 
If $x=1$ we have $y=\frac{r^2}{2h};$ and if $x\ne 1$ then, using $y\in[1,h]$,
\begin{align*}
y& =\frac{-h + \sqrt{h^2+(x^2-1)r^2}}{x^2-1} = \frac{-h + x\sqrt{r^2-1}}{x^2-1}
\\	& =\frac{-h^2+x^2(r^2-1)}{(x^2-1)\bigl(h+x\sqrt{r^2-1}\bigr)}
=\frac{(x^2-1)r^2}{(x^2-1)\bigl(h+x\sqrt{r^2-1}\bigr)}.
\end{align*}

Consequently (for both $x = 1$ and $x \ne 1$),
\[
\frac{y}{h}
=\frac{r^2}{h\bigl(h+x\sqrt{r^2-1}\bigr)}
=\frac{r^2}{h^2+h\sqrt{r^2-1}\sqrt{r^2-h^2}}.
\]
Since $r \leq \sqrt{d}$, it suffices to show that for all $h\in[1,r]$,
\[	\frac{y}{h} \ge
	\frac{2}{1+r}.
\]

Set
\[
t := \frac{\sqrt{r^{2}-h^{2}}}{h} \quad \text{and} \quad k := \sqrt{r^{2}-1}.
\]
Then $\sqrt{r^{2}-h^{2}}=ht$ and $r^{2}=h^{2}(1+t^{2}).$
Hence,
\[
\frac{y}{h}  
= \frac{h^{2}(1+t^{2})}{h^{2}(1+kt)}
=\frac{1+t^{2}}{1+kt}.
\]
So the desired inequality is equivalent to
\begin{align*}
	\frac{1+t^{2}}{1+kt}\ge \frac{2}{1+r}
	\quad\Longleftrightarrow\quad& (1+r)(1+t^{2})-2(1+kt)\ge 0
	\\ \quad\Longleftrightarrow\quad&	(r+1)t^{2}-2kt+(r-1) \ge 0
	\\ \quad\Longleftrightarrow\quad&	\bigl(\sqrt{r+1}\,t-\sqrt{r-1}\bigr)^{2}\ge 0.
	\qedhere
\end{align*}	
	
\end{proof}

\begin{remark}
	Lemma \ref{lem:john-radial} is sharp. The lower bound is attained when $r = \sqrt{d}$ which holds for cubes, and $u$ is such that $x = \sqrt{\frac{r(r-1)}{2}}$ and $h = \sqrt{\frac{r(r+1)}{2}}$.
\end{remark}

\begin{remark}
	The lower bound in Theorem~\ref{th:two-over-sqrtd} can be sharpened. Its proof
	works verbatim with Banach--Mazur distance $d_{\mathrm{BM}}(L,\B)$ in place of $\sqrt d$. 
	Likewise, Minkowski--Radon argument works verbatim
	with the Minkowski measure of symmetry
	$m^*_K$ in place of $d$ (see \cite{toth2015measures} for definition and properties). Consequently,
	\[
	\sum_i \rw_K(P_i)
	\geq
	\frac{2}{1+\min\{d_{\mathrm{BM}}(L,\B),m^*_K\}} .
	\]
	
	In the planar case, this allows one to improve the bound
	$
	\frac{2}{1+\sqrt2} \approx 0.828
	$
	to
	approximately $0.836$.
	To upper bound $\min\{d_{\mathrm{BM}}(L,\B),m^*_K\}$,
	one can relate both quantities to Banach--Mazur distance from $L$ to a square. The main ingredient for it is Theorem~1.3 from \cite{grundbacher2026certain}.
\end{remark}

\bigskip 

\providecommand{\bysame}{\leavevmode\hbox to3em{\hrulefill}\thinspace}
\providecommand{\MR}{\relax\ifhmode\unskip\space\fi MR }
\providecommand{\MRhref}[2]{%
  \href{http://www.ams.org/mathscinet-getitem?mr=#1}{#2}
}
\providecommand{\href}[2]{#2}

\end{document}